\documentclass[12pt,a4paper]{amsart}
\usepackage[english]{babel}
\usepackage[utf8]{inputenc}
\usepackage{tikz}
\usepackage{lipsum}
\usepackage{amssymb}
\usepackage{float}
\newtheorem{Theorem}{Theorem}[section]

\newtheorem{Lemma}[Theorem]{Lemma}

\newtheorem{Proposition}[Theorem]{Proposition}
\newtheorem{Definition}[Theorem]{Definition}
\newtheorem{Corollary}[Theorem]{Corollary}
\newtheorem{Remark}[Theorem]{Remark}
\newtheorem{Proof of Theorem 3.1}[Theorem]{Proof of Theorem 3.1}
\newcommand{\s}{\underline{s}}
\newcommand{\Di}{\Delta_{\mathrm{inv}}}
\usepackage[toc,page,header]{appendix}

\title[Boolean Involutions]{Boolean Complexes of Involutions}

\author{Axel Hultman \and Vincent Umutabazi}

\address{Department of Mathematics, Link\"oping University, SE-581 83 Link\"oping, Sweden}

\email{axel.hultman@liu.se}
\email{vincent.umutabazi@liu.se}

\begin{document}

\begin{abstract}
Let $(W,S)$ be a Coxeter system. We introduce  the boolean 
 complex of involutions of $W$  which is an analogue of the boolean complex of 
 $W$ studied by Ragnarsson and Tenner.
 By applying discrete Morse theory, we determine the homotopy type of the boolean 
 complex of involutions for a large class of $(W,S)$, including all finite Coxeter groups, 
 finding that the homotopy type is that of a wedge of spheres of dimension $\vert S\vert-1$.
 In addition, we find simple recurrence formulas for the number of spheres in the wedge.

\end{abstract}
\maketitle

\section{Introduction} \label{sec:Introduction}

Let $F$ be a finite alphabet. A word over $F$ is {\em injective} if it is a sequence of distinct letters from $F$. The poset of all injective words over $F$ ordered by the (not necessarily consecutive) subword relation is a simplicial poset, hence the face poset of a boolean cell complex. This is the well-studied {\em complex of injective words}. Farmer \cite{MR529895} showed that it is homotopy equivalent to a wedge of spheres of dimension $|F|-1$. Jonsson and Welker \cite{MR2552261} proved the same conclusion for considerably more general classes of boolean cell complexes using shellability techniques.

Ragnarsson and Tenner \cite{MR2538015} independently arrived at one of Jonsson and Welker's families of complexes from a different starting point. Namely, Ragnarsson and Tenner considered the {\em boolean elements} of a Coxeter system $(W,S)$, i.e.\ those elements whose Bruhat order ideals are boolean lattices. These elements were first considered by Tenner in \cite{MR2333139}. Under Bruhat order, they generate the face poset of a boolean cell complex, the {\em boolean complex of $(W,S)$}. Using discrete Morse theory, it was shown in \cite{MR2538015} that this complex is homotopy equivalent to a wedge of spheres of dimension $|S|-1$, and a recurrence formula for the number of spheres was found.

In \cite{MR2519850}, the first named author and Vorwerk studied the {\em boolean involutions} of $(W,S)$ as an involution analogue of Tenner's boolean elements. In the Bruhat order on involutions, they too form the face poset of a boolean cell complex. In this paper, we study this {\em boolean complex of involutions}, $\Delta_{\mathrm{inv}}(W)$. Such complexes are not special cases of Jonsson and Welker's constructions.  

For a large class of Coxeter systems, including all finite ones, we apply discrete Morse theory in order to show that $\Delta_{\mathrm{inv}}(W)$ is also homotopy equivalent to a wedge of spheres of dimension $|S|-1$ and provide a recurrence formula for the number of spheres.

The remainder of this paper is organized in the following way. Section \ref{sec:prel} contains preliminaries on cell complexes, acyclic matchings, and Coxeter group theory. Section \ref{sec:main} contains the main technical result of this paper, Theorem \ref{th:mai2}. In Section \ref{sec:mainn}, we provide applications of the main result by computing the homotopy type of $\Delta_{\mathrm{inv}}(W)$ for many Coxeter groups, including all finite types.

\section{Preliminaries}\label{sec:prel}
In this section, notions on simplicial posets, acyclic matchings and  cell complexes
are collected for being used in the sequel.
\begin{center}
	
\end{center}
\subsection{Posets, matchings, and cell complexes}\label{sec:comp}
Let $P$ be a poset and $x,y\in P$ such that $x<y$. We say that 
$y$  \textit{covers} $x$, written as $x\lessdot y$, if there is 
no $z\in P$ such that $x<z<y$.
\begin{Definition}
Let $P$ be a poset with cover relation $\lessdot$. An involution 
$ M:P\rightarrow P $ (i.e., $M\circ M=id$) such that for each 
$ x\in P $, $x\lessdot M(x)$ or $M(x)\lessdot x$ or $M(x)=x$ is 
called a \emph{matching}. An element $x\in P$ such that $M(x)=x$ 
is called \emph{critical}. 
\end{Definition}
\begin{Definition}\label{Def2}
Let $P$ be a  poset and $M$ be a matching on $P$. 
Consider the Hasse diagram of $P$ as a directed graph with edges
$x\leftarrow y$ if $x\lessdot y$ for $x,y\in P$. From the Hasse 
diagram of $P$ construct a new directed graph $G_{M}(P)$ by 
reversing each arrow $x\leftarrow y $ to $x\rightarrow y$ if 
$y=M(x)$. The matching $M$ on $P$ is called \emph{acyclic} if 
there are no directed cycles in  $G_{M}(P)$.
\end{Definition}	
In $G_{M}(P)$, the unchanged edge $x\leftarrow y $  is said to 
be \emph{downward} while 
the edge $x\rightarrow M(x)$ is said to be \emph{upward}.
Since for any matching $M$, no  two incident edges are directed upward 
 in $G_{M}(P)$, the next lemma follows:	
\begin{Lemma}\label{LL1}
If a matching $M$ on a poset $P$ is not acyclic, then any directed cycle 
in $G_{M}(P)$ is of the form 
 $$ a_{1}\lessdot M(a_{1}) \gtrdot a_{2} \lessdot 
M(a_{2}) \ldots 
\gtrdot  a_{k} \lessdot M(a_{k} )\gtrdot  a_{1}, $$ 
 for pairwise distinct $ a_{i}$ and some $k\geq 2$.
\end{Lemma}	

Let $\Delta$ be a finite regular cell complex.
The \emph{face poset} of $\Delta$, denoted by $P(\Delta)$, is the 
poset of all cells ordered by set inclusion of their 
closures where a minimum 
element (sometimes thought of as the empty cell) is also in $P(\Delta)$.
For more on regular cell complexes, see \cite{MR3822092}.

\begin{Definition}
Let $\Delta$ be a finite regular cell complex and $P(\Delta)$ its face poset.
 Then  $\Delta$ is called a \emph{boolean cell complex} if $P(\Delta)$ 
 is a \emph{simplicial} poset (i.e., a poset with a minimum element in which 
 every principal order ideal is isomorphic to a boolean lattice).
\end{Definition}	

Conversely, any finite simplicial poset is the face poset of a boolean cell complex which is unique (up to cellular isomorphism).

The following is one of the most useful consequences of discrete Morse theory. It will serve as our main topological tool.
\begin{Theorem}\cite[Theorem 6.3 ]{MR1939695}\label{T1}
Let $\Delta$ be a boolean cell complex and let $M$ 
be an acyclic matching on the face poset $P(\Delta)$. 
If there are $c$ critical cells, all of the same dimension $m$,
 then $\Delta$ is homotopy equivalent to 
 a wedge of $c$ spheres of dimension $m$.
\end{Theorem}
\begin{Corollary}
If an acyclic matching $M$ on  $P(\Delta)$ does not have any critical 
cells, then $\Delta$ is contractible. 
\end{Corollary}

Lemma \ref{L1} has been discovered by Jonsson \cite{MR2018427} 
and independently by Hersh \cite{{MR2164921}}. It provides a way to patch together acyclic matchings.

\begin{Lemma}\label{L1}
Let $\Delta$ be a boolean cell complex such that 
$\Delta=\bigcup_{\alpha\in Q} \Delta_{\alpha}$ 
where the index set  $Q$ is partially ordered. Assume that:
\begin{enumerate}
\item Each cell $\gamma\in \Delta$ is only in one $\Delta_{\alpha}$ 
(i.e., the union is disjoint).
\item $\bigcup_{\beta \leq \alpha \in Q}\Delta_{\beta}$ 
is an order ideal of $P(\Delta)$ for every $\alpha \in Q$.
\end{enumerate}
For each $\alpha \in Q $, suppose that  $M_{\alpha}$ is an acyclic matching
on $P(\Delta_{\alpha})$ where $P(\Delta_{\alpha})$ is the subposet of 
$P(\Delta)$ induced by $\Delta_{\alpha}$.
Then $\bigcup_{\alpha\in Q}M_{\alpha}$ is an acyclic matching on
$P(\Delta)$.
\end{Lemma} 
\smallskip
\subsection{Coxeter systems}\label{sec:cox}
We review some facts on Coxeter groups for later use. For more, 
we refer  the reader to \cite{MR2133266} or \cite{MR1066460}. 

Let $W$ be a group generated by a finite set of involutions $S$, where 
$ (ss')^{m(s,s')}=e $ for $ m(s,s)=1 $, $m(s,s')=m(s',s)\geq 2$ 
for all $ s\neq s' $ in $S$ and $ m(s,s')\in \mathbb{Z}^{+}\cup \{\infty\} $. 
Here, $e$ denotes the identity element in $W$.
The group $W$ as described above, together with its generating set $S$ form a pair 
$(W,S)$ which is called a \emph{Coxeter system}. Any $w\in W$ is a product of 
generators in $ S $. That is $ w=s_{1}s_{2}\cdots s_{k} $ for some $s_{i}\in S$. 
If $k$ is smallest among all such expressions for $w$, then $k$ 
is called the \emph{length} of $w$, denoted $\ell(w)=k$, and $s_{1}s_{2}\cdots s_{k}$ 
is called a \emph{reduced expression} for $w$. Note that we do not distinguish 
notationally between an expression and the group element it represents.

A Coxeter system $(W,S)$  is sometimes represented by its \emph{Coxeter graph}. 
It is the edge-labeled (simple, undirected) graph with vertex set  $S$ 
and an edge labeled 
$m(s,s')$  connecting two vertices $s$ and $s'$ if $m(s,s')\geq 3$.  
By convention, the label is usually omitted if it is equal to 3, 
see Figure \ref{figs:d}.

  If $W$ is finite, it has a unique longest element $w_{0}$. 
 A subgroup $W_{J}$ of $W$, generated by $J\subseteq S$ is called 
\emph{parabolic}. If $W_{J}$ is finite, its longest element is
denoted by $w_{0}(J)$.

A Coxeter system is \emph{irreducible} if its Coxeter graph is connected.
The Coxeter graphs in Figure \ref{figs:d} indicate the \emph{classification} 
of finite irreducible Coxeter groups. 
There are  three classical families of types
$A_{n}$ $(n\geq 1)$, $B_{n}$ $(n\geq 2)$, $D_{n}$ $(n\geq 4)$, 
six exceptional
 groups 
of types $E_{6}$, $E_{7}$, $E_{8}$, $F_{4}$, $H_{3}$ and  $H_{4}$, 
 and
 one family of dihedral groups of type $ I_{2}(m)$
$(m\geq 3)$.

\begin{figure}[ht]
		\includegraphics[width=\linewidth]{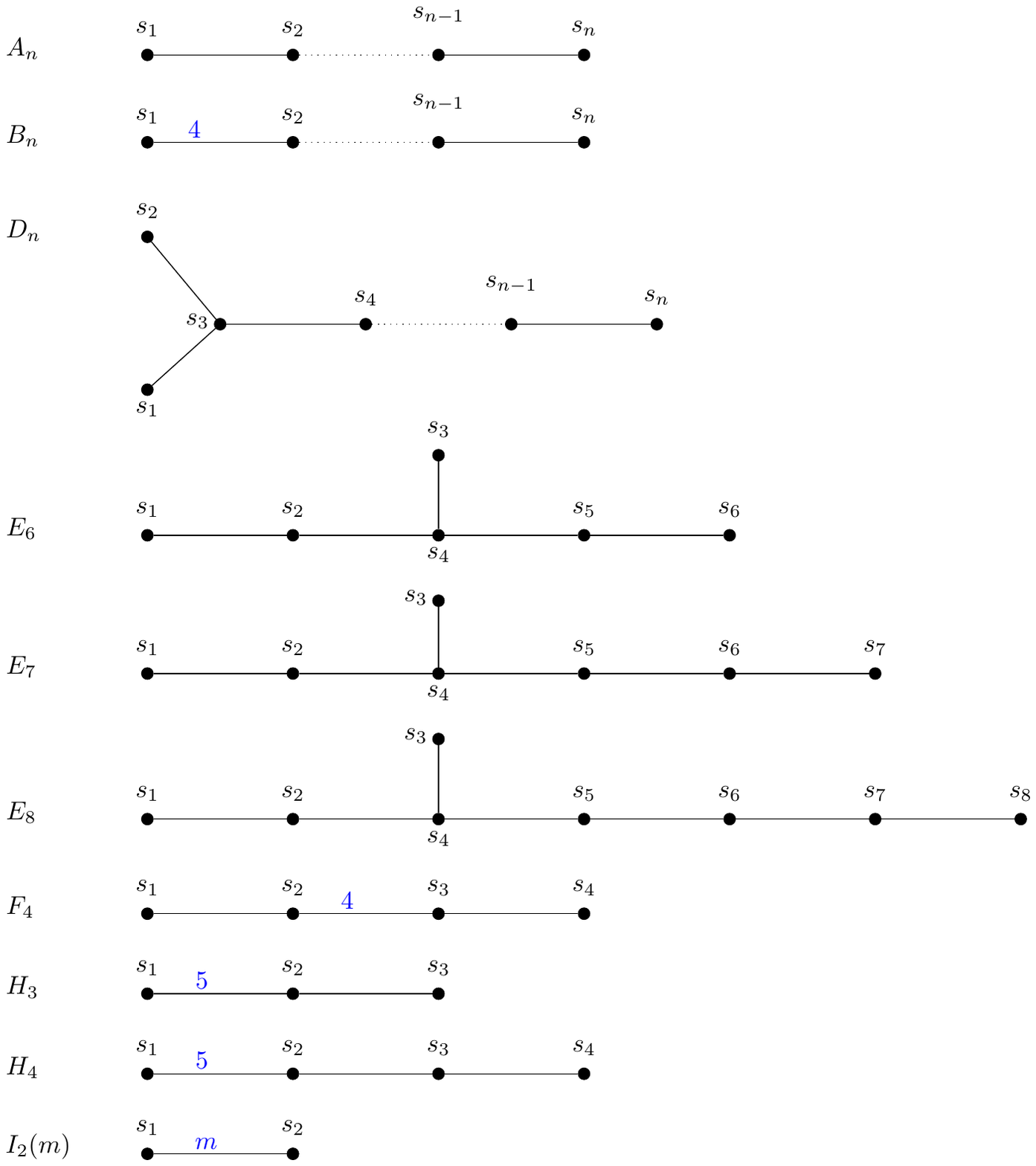}
		\vspace{-5cm}
		\caption{The finite, irreducible Coxeter groups.}
		\label{figs:d}
	\end{figure}	
\setlength\belowcaptionskip{-3ex}

Given $w\in W$, define the \emph{right descent} set of $w$ by  
$$D_{R}(w)=\{s\in S| \ell(ws)<\ell(w)\}.$$
For $s\in S$,  $s \in D_{R}(w)$ if and only if $w$ has a reduced expression that ends in $s$.

Let $T:=\{wsw^{-1}| w\in W, s\in S\}$ be the set of \emph{reflections} in $W$.
\begin{Definition}
 For $u,w\in W$, let $u\rightarrow w$ if there is $t\in T$ such that $w=ut$ and 
 $\ell (u)< \ell(w)$. The \emph{Bruhat order} is the partial order relation on $W$ 
 which is defined by 
$u\leq w$ if there is a sequence 
$$u=u_{0}\rightarrow u_{1}\rightarrow \cdots u\rightarrow u_{n}=w.$$
\end{Definition}

\begin{Definition}
	Let $(W,S)$ be a Coxeter system and consider $W$ as a poset under Bruhat order.
	An element $w\in W$ is called \emph{boolean} if its lower principal order ideal 
	 is isomorphic to a boolean algebra. 
\end{Definition}
The following lemma is a characterization of boolean elements. 
For $W$ of type $A$, $B$, or $D$ this lemma is \cite[Proposition 7.3]{MR2333139}.
\begin{Lemma}\cite{MR2538015}\label{T2}
	Let $(W,S)$ be a Coxeter system and $w\in W$. 
	Then $w$ is boolean if and only if it has some reduced expression 
	without repeated letters.
\end{Lemma}
In fact, Lemma \ref{T2} is equivalent to the statement that  ``$w$  is  boolean
if and only if \emph{no} reduced expression for $w$ has repeated letters.''
Let $\mathbb{B}(W)$ denote the subposet of $W$ induced by 
the boolean elements. The subposet $\mathbb{B}(W)$ 
is the \emph{boolean ideal}. Note that $\mathbb{B}(W)$ is a 
simplicial poset.
\begin{Definition}
	For a Coxeter system $(W,S)$, the \emph{boolean complex}  
	of $(W,S)$ 
	is the boolean cell complex $\Delta(W)$ whose face poset is
	$\mathbb{B}(W)$.
\end{Definition}
All maximal cells of $\Delta(W)$  have dimension $ |S|-1$.
Boolean complexes of Coxeter groups were studied by 
Ragnarsson and Tenner \cite{MR2538015}. 
They showed that for every Coxeter system $(W,S)$, there is 
a nonnegative integer $\beta(W)$ called \emph{boolean number} 
such that $\Delta(W)$ is homotopy equivalent to a wedge 
of $\beta(W)$ spheres of 
dimension $\vert S \vert-1$. 
Moreover, they found a recurrence formula for computing $\beta(W)$ in 
terms of the (unlabeled) Coxeter graph. Note that the labels of the Coxeter 
graph have no effect in determining $\beta(W)$. 
That is, if $W$ and  $W'$ have isomorphic Coxeter graphs when ignoring all edge 
labels that are at least 4, then $\mathbb{B}(W) \cong \mathbb{B}(W')$. In particular, $\beta(W)=\beta(W')$ \cite{MR2538015}.
\subsection{Reduced $\underline{S}$-expressions}\label{sec:red}
In this subsection we review some material on reduced $\underline{S}$-expressions 
and boolean involutions needed in Sections \ref{sec:main} and \ref{sec:mainn}.

For a Coxeter system $(W,S)$, let $I=\{w\in W| w=w^{-1}\}$ denote the set of involutions. Consider an alphabet $\underline{S}=\{\underline{s} \: | \: s\in S\}$. 
\begin{Definition}\label{Dd}
The free monoid
$\underline{S}^{\ast}$ acts from the right on the set
$W$ by
$$w\underline{s}=\begin{cases}
	 ws \: \: \:\: \text{if}\: \: \: sws=w,\\
	sws\: \: \: \text{otherwise},
\end{cases}$$
and $w\underline{s}_{1}\underline{s}_{2}\cdots \underline{s}_{k}=
(\cdots ((w\underline{s}_{1})\underline{s}_{2})\cdots \underline{s}_{k})$.
\end{Definition}

Note that $w\underline{s}\,\underline{s}=w$ for all $w\in W$ and $s\in S$.
For the identity element $e\in W$, write
$\underline{s}_{1}\underline{s}_{2}\cdots \underline{s}_{k}$ for
$e\underline{s}_{1}\underline{s}_{2}\cdots \underline{s}_{k}$.

 For every $w\in I$ there are symbols $\underline{s}_{1}, 
\underline{s}_{2},\ldots,\underline{s}_{k}$ such that $w=\underline{s}_{1}
\underline{s}_{2}\cdots\underline{s}_{k}$ and, conversely,  all such 
elements are involutions \cite{MR1066573}. The minimal $k$ such that
$w=\underline{s}_{1}\underline{s}_{2}\cdots \underline{s}_{k}$ for some
$\underline{s}_{1}, \underline{s}_{2}, \ldots,  \underline{s}_{k}\in
\underline{S}$ is called the $rank$ of $w$, and is denoted $\rho(w)$.
The expression $\underline{s}_{1}\underline{s}_{2}\cdots \underline{s}_{k}$
is then called a \textit{reduced} $\underline{S}$\textit{-expression} for $ w$.

Reduced $\underline{S}$-expressions appears under various names in the literature. It is (the right-handed version of) what Richardson and Springer call ``admissible sequences'' \cite{MR1066573} and Hamaker, Marberg, and Pawlowski \cite{MR3627501} refer to as ``involution words''. In this paper the notation 
is taken from \cite{MR2286056, MR2430307}.

\begin{Definition}
Let $\alpha_{\underline{s},\underline{s}'}=\underline{s}\,\underline{s}'
\underline{s}\cdots$ be the alternating word in  
$\underline{s}$ and $\underline{s}'$ where the number of letters 
$\underline{s}$ and $\underline{s}'$  in 
$\alpha_{\underline{s},\underline{s}'}$ is equal to $m(s,s')$. 
Operating on words in $\underline{S}^{\ast}$, 
we define a \emph{braid move} as the replacement of $\alpha_{\underline{s},\underline{s}'}$ 
by $\alpha_{\underline{s}',\underline{s}}$.
\end{Definition}
In fact braid moves in a reduced $\underline{S}$-expression for 
$w\in I$ preserve $w$ (see  \cite{MR3627501}).
\begin{Definition}
Let $\mathcal{\hat{R}}(w)$ denote the set of all
 reduced $\underline{S}$-expressions for $w\in I$. 
For $u,w\in I$, the replacement of one element 
in $\mathcal{\hat{R}}(u)$ by another in the beginning 
of a reduced $\underline{S}$-expression for $w\in I$ is called an 
\emph{initial move}. 
\end{Definition}
The following theorem provides a minimal list of initial moves that 
combined with braid moves suffice to connect all elements of 
$\mathcal{\hat{R}}(w)$.
\begin{Theorem} \cite{MR3861777}\label{BM}
Let $(W, S)$ be a Coxeter system and $w\in I$. Any two reduced
$\underline{S}$-expressions for $w$ can be connected by a 
sequence of braid moves and initial moves that replace $u$ with 
$v$ if  $u, v \in \mathcal{\hat{R}}(w_{0}(J))$ for some  
$J \subseteq S$. 
The following parabolic subgroups $W_{J}$ are necessary and sufficient:
\begin{enumerate}
  \item $W_{J}$ of type $B_{3}$;
  \item $W_{J}$ of type $D_{4}$;
  \item $W_{J}$ of type $H_{3}$;
  \item $W_{J}$ of type $I_{2}(m)$, for $m\geq 3$.
\end{enumerate}
\end{Theorem}
In other words, given $x,y\in \mathcal{\hat{R}}(w)$, there exist some $x_{0}, x_{1}, \cdots, x_{k}$ such that $x=x_{0}, x_{k}=y$ 
and for all $i$, $x_{i}$ differs from $x_{i+1}$ by a braid move or an initial move of the kind specified in Theorem \ref{BM}.
\begin{Lemma}\cite{MR2286056}\label{Bl}
Let $w\in I $.  Then $s\in D_{R}(w)$
if and only if $w$ has a reduced $\underline{S}$-expression 
that ends in $\underline{s}$. Moreover, $\rho(w\underline{s}) = \rho(w)+1$ if and only if $s\not \in D_R(w)$.
\end{Lemma}

Let $\mathrm{Br}(I)$ denote the subposet of the Bruhat order on $W$ induced by $I$. 
The order relation on 
 $\mathrm{Br}(I)$ has been characterized in \cite{MR2286056, MR1066573, MR1266276}.
 
\begin{Lemma}[Subword property of $I$] \cite{MR2286056}
Suppose that $v,w\in I$ and 
$\underline{s}_{1}\underline{s}_{2}\cdots \underline{s}_{k}$ is 
a reduced $\underline{S}$-expression for $w$. Then $ v\leq w $ if and only if 
$ v=\underline{s}_{i_{1}}\underline{s}_{i_{2}}\cdots \underline{s}_{i_{r}}  $ for 
some $ 1\leq i_{1}\leq i_{2}\leq \cdots \leq i_{r}\leq k $.
\end{Lemma} 
 For $w\in I$, let $B(w)$ denote the principal order ideal generated by $w$ in $\mathrm{Br}(I)$. 
 Then we have:
\begin{Definition}
An element $w\in I$ is called a \emph{boolean involution} if 
$B(w)$ is isomorphic to a boolean algebra.
\end{Definition}
Analogously to the characterization of boolean elements (i.e., Theorem \ref{T2}),
 we have:
\begin{Proposition}\cite{MR2519850}\label{P1}
Let $w\in I$. Then the following are equivalent:
\begin{enumerate}
\item $w$ is a boolean involution.
\item $w$ has some reduced $\underline{S}$-expression without 
repeated letters. 
\end{enumerate}
\end{Proposition}
In fact, Proposition \ref{P1} is equivalent to ``$w$ is 
a boolean involution if and only if \emph{no} reduced 
$\underline{S}$-expression for $w$ has repeated letters.''

Suppose $s,s'\in S$ satisfy $m(s,s')=3$. Let us say that a 
\emph{half-braid move} is the replacement of of 
$\underline{s}\,\underline{s}'$ by $\underline{s}'\underline{s}$ 
in the beginning of a reduced $\underline{S}$-expression. 
\begin{Proposition}\label{P12}
Any two reduced $\underline{S}$-expressions for a boolean involution 
$w\in I$ can be connected by a sequence of braid moves and 
half-braid moves.
\end{Proposition}
\begin{proof}
In the list of parabolic subgroups given in Theorem \ref{BM}, 
it is only in type $I_2(3)$ that the longest element is a boolean involution. 
It has only two reduced $\underline{S}$-expressions, and the initial move 
replacing one of them by the other is a half-braid move.   
\end{proof}

Let $P(\Di(W))$ denote the subposet of $\mathrm{Br}(I)$ induced by
all boolean involutions. 
It is a simplicial poset which we call the \emph{ boolean involution ideal}. 
We have the following definition:
\begin{Definition}\label{D1} Let $ \Di(W) $
be the boolean cell complex whose face poset is $P(\Di(W))$.
\end{Definition}
By abuse of notation, we identify a boolean 
involution with the cell it represents. 	
It follows from Proposition \ref{P12} together with Lemma \ref{Bl} that every $\underline{S}$-expression without repeated letters is reduced. Hence, every maximal cell in $\Di(W)$ has dimension $\vert S\vert-1$.
Let $\beta_I(W)$ denote the absolute value of the reduced Euler
characteristic of $ \Di(W) $.

In contrast to the situation for the boolean complex $\Delta(W)$, it is not the case that $\Di(W)$ only depends on the unlabeled Coxeter graph. As the following lemma shows, it is however not necessary to keep track of the actual labels:
\begin{Lemma}\label{le:r}
If the Coxeter graph of a Coxeter system $(W',S)$ is obtained from that of another system $(W,S)$ by increasing some edge labels which are at least $4$, then $\Di(W') \cong \Di(W)$.
\begin{proof}
In a reduced $\underline{S}$-expression without repeated letters, the possible braid moves are provided by the vertices that are not adjacent in the Coxeter graph, and the half-braid moves come from the vertices that are connected by an edge labeled $3$ (i.e., with omitted label). It follows that the boolean involution ideals of $W$ and $W'$ are isomorphic.
\end{proof}
\end{Lemma}

\section{ Acyclic matchings on the boolean involution ideal} \label{sec:main}

Say that a Coxeter system $(W,S)$ (with $S$ finite) is {\em ordered} if the generating 
set $S$ is endowed with a total order. For such a system, with $s = \max S$, define
\[
P_1(W) := \{a\in \Di(W)\mid\ \s \not \le a\text{ or }s\in D_R(a)\}.
\]
Then, let
\[
\begin{split}
\Gamma(W) :&= \Di(W)\setminus P_1(W)\\
 &= \{a\in \Di(W)\mid \s \le a\text{ and }s\not \in D_R(a)\}.
\end{split}
\]
\begin{Definition}\label{Gamma}
Suppose $(W,S)$ is an ordered Coxeter system. A matching $M$ 
on $P(\Di(W))$ is a {\em $\Gamma$-matching} if it satisfies the following properties:
\begin{itemize}
\item $M$ is acyclic;
\item All critical cells of $M$ are of top dimension $|S| - 1$;
\item $M$ preserves $\Gamma(W)$;
\item All critical cells of $M$ belong to $\Gamma(W)$.
\end{itemize}
\end{Definition}
In particular, by Theorem \ref{T1}, if $P(\Di(W))$ has a 
$\Gamma$-matching, then  $ \Di(W) $ is homotopy equivalent to a 
wedge of $ \beta_I(W) $ spheres of dimension $\vert S \vert-1$.

If $(W,S)$ is ordered, we let $W_{-k}$ denote the parabolic subgroup 
which is generated by the first $|S|-k$ elements of $S$, i.e.\ the $k$ largest 
elements are removed from $S$. The corresponding subsystem is ordered 
with the order inherited from $S$.

\begin{Definition}\label{Path}
Let $(W,S)$ be an ordered Coxeter system with $S =\{s_1<\cdots<s_n\}$ where $n\geq 3$. 
We say that $(W,S)$ is {\em path ended} if $s_n$ commutes 
with every $s_i$ for $i\le n-2$, $s_{n-1}$ commutes with every 
$s_i$ for $i\le n-3$, $s_ns_{n-1}s_n=s_{n-1}s_ns_{n-1}$ and 
$s_{n-1}s_{n-2}s_{n-1}=s_{n-2}s_{n-1}s_{n-2}$. 
\end{Definition}

Definition \ref{Path} means that the Coxeter graph of a path ended system $(W,S)$ is 
as illustrated in Figure \ref{fi:bg(Wn)}, where the parabolic subgroups $W_{-2}$ and $W_{-3}$ are
 also indicated; $W_{-2}$ is generated by all $s_{i}$ inside the blue circle, 
 and $W_{-3}$ is generated by all $s_{i}$ inside the black circle, respectively.
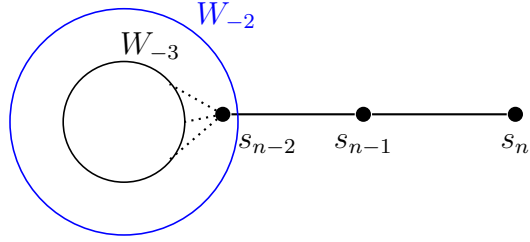
\begin{figure}[ht]
	\centering
	\begin{tikzpicture}[
	thick,
	acteur/.style={
		circle,
		fill=black,
		thick,
		inner sep=0.5pt,
		minimum size=0.2cm
	}
	]
	\node (a1) at (0,0.60) [acteur,label=below:$\;\;\;\;\;\;\; \; \; \; s_{n-2}$]{};
	\node (a2) at (1.85,0.60) [acteur,label=below:$s_{n-1} $]{};
	\node (a3) at (3.85,0.60) [acteur,label=below:$s_{n} $]{};
	\draw[-] (a1) -- (a2);
	\draw[-] (a2) -- (a3);
	\draw [semithick,black] (-1.3,0.5) circle (0.8);
	\coordinate [label=right:\textcolor{blue}{$W_{-2}$}] ($W_{n-2}$) at (-0.5,1.9);
	\draw [semithick,blue] (-1.3,0.5) circle (1.5);
	\coordinate [label=right:\textcolor{black}{$W_{-3}$}] ($W_{n-2}$) at (-1.5,1.50);
	\draw [dotted] (0,0.60) -- (-0.7,1);
	\draw [dotted] (0,0.60) -- (-0.5,0.5);
	\draw [dotted] (0,0.60) -- (-0.7,0.0);
	\end{tikzpicture}
	\caption{A Coxeter graph of a path ended system $(W,S)$} \label{fi:bg(Wn)}
\end{figure}
\setlength\belowcaptionskip{-3ex}
$$\,$$
Now we have the following theorem.
\begin{Theorem}\label{th:main}
Let $(W,S)$ be a path ended Coxeter system. Assume that $P(\Di(W_{-2}))$ 
and $P(\Di(W_{-3}))$ have {\em $\Gamma$-matchings} $M^{-2}$ and $M^{-3}$, respectively. 
Then, $P(\Di(W))$ has a {\em $\Gamma$-matching} $M$. Moreover, the number 
of critical elements of $M$ equals the number of critical elements of $M^{-2}$
 plus the number of critical elements of $M^{-3}$.
\end{Theorem}

The proof of Theorem \ref{th:main} is deferred to the next subsection.
 Here we proceed to apply it in order to prove Theorem \ref{th:mai2} below, 
 which is the main result of this paper.

For two ordered Coxeter systems $(W,S)$ and $(Z,S')$, 
we say that $(Z,S')$ is a {\em path extension} of $(W,S)$ if 
$S=\{s_1<\cdots<s_n\}$ and $S'=\{s_1<\cdots<s_N\}$ 
for some $N\ge n$, and for all $1\le i\le N$ and $n<j\le N$ with $i\neq j$, it holds that
\[
m(s_i,s_j)=\begin{cases}
2 & \text{if $|i-j|\ge 2$,}\\
3 & \text{if $|i-j| = 1$.}
\end{cases}
\]

In other words, we obtain the Coxeter graph of a path extension of $(W,S)$ 
from that of $(W,S)$ by attaching a (possibly empty) path at the maximum vertex $s_n$. Figure \ref{fi:bg(Z)} illustrates this situation for a path ended system $(W,S)$. This is the setting of the upcoming theorem.
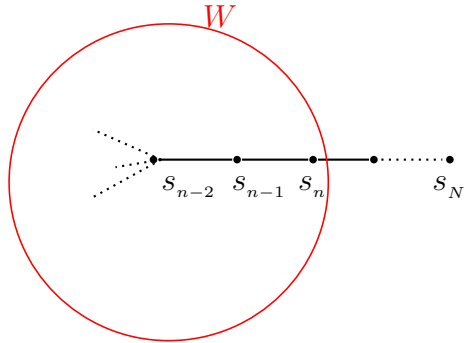
\begin{figure}[ht]
	\centering
	\begin{tikzpicture}[
	thick,
	acteur/.style={
		circle,
		fill=black,
		thick,
		inner sep=0.5pt,
		minimum size=0.1cm
	}
	]
	\node (a1) at (-0.1,0.60) [acteur,label=below:$\;\;\;\;\; \; \; \; s_{_{n-2}}$]{};
	\node (a2) at (1,0.60) [acteur,label=below:$\;\;\;\;\;s_{_{n-1}}$]{};
	\node (a3) at (2,0.60) [acteur,label=below:$s_{_{n}} $]{};
	\node (a4) at (2.8,0.60) [acteur,label=below:]{};
	\node (a5) at (3.8,0.60) [acteur,label=below:$s_{_{N}} $]{};
	\draw[-] (a1) -- (a2);
	\draw[-] (a2) -- (a3);
	\draw[-] (a3) -- (a4);
	\draw [dotted] (2.8,0.60) -- (3.7,0.60);
	\draw [semithick,red] (0.1,0.3) circle (2.1);
	\coordinate [label=right:\textcolor{red}{$W$}] ($W$) at (0.4,2.5);
	\draw [dotted] (0,0.60) -- (-0.9,1);
	\draw [dotted] (0,0.60) -- (-0.6,0.5);
	\draw [dotted] (0,0.60) -- (-0.9,0.1);
	\end{tikzpicture}
	\caption{\small{A Coxeter system $(Z,S')$ as a path extension of a path ended Coxeter system $(W,S)$.} \label{fi:bg(Z)}}
\end{figure}
$$ \, $$

\begin{Theorem} \label{th:mai2} 
Let $(W,S)$ be a path ended Coxeter system such that $P(\Delta_{\mathrm{inv}}(W_{-1}))$, 
$P(\Delta_{\mathrm{inv}}(W_{-2}))$ and $P(\Delta_{\mathrm{inv}}(W_{-3}))$ all have $\Gamma$-matchings. 
Then, for every path extension $(Z,S')$ of $(W,S)$, $\Delta_{\mathrm{inv}}(Z)$ is homotopy equivalent to a wedge 
of $(|S'|-1)$-dimensional spheres. The number of spheres satisfies the recurrence 
$\beta_I(Z) = \beta_I(Z_{-2}) + \beta(Z_{-3}).$
\end{Theorem}
\begin{proof}
Since $(Z,S')$ is path ended, Theorem \ref{th:main} together with Theorem \ref{T1} provide 
the desired conclusions if we are able to show that $Z_{-2}$ and $Z_{-3}$ have $\Gamma$-matchings. 
By induction on $|S'|$, it however follows from Theorem \ref{th:main} that $Z_{-1}, Z_{-2}$, and $Z_{-3}$ all have $\Gamma$-matchings.
\end{proof}

\subsection{Proof of Theorem \ref{th:main}}
We now proceed to prove Theorem \ref{th:main}. 
First let us establish a couple of lemmas.
 
 Let $(W,S)$ be as in Definition \ref{Path}.  
\begin{Lemma}\label{L111}
Let $w\in W$ be a boolean involution. 
If $\underline{s}_{n-1}$ comes after both $\underline{s}_{n-2}$ and
$\underline{s}_{n}$ in some reduced $\underline{S}$-expression for $w$,
then this holds in every reduced $\underline{S}$-expression for $w$.
\end{Lemma}

\begin{proof}
Note that if $\s_{n-1}$ comes after both $\s_{n-2}$ and $\s_{n}$ in a 
reduced $\underline{S}$-expression for $w$, then this cannot be changed by a 
half-braid move or a braid move. Hence the result follows from 
Proposition \ref{P12}.

\end{proof}

Now define \[P_2(W):=\{x\underline{s}_{n-2}
\underline{s}_{n}\underline{s}_{n-1}: x\in \Di(W_{-3})\}.\]
\begin{Lemma}\label{C2}
	We have $ P_2(W) \subseteq \Gamma(W)$.	
\end{Lemma}	
\begin{proof}
	By Lemma \ref{L111}, $ \underline{s}_{n-1} $	comes after
	$ \underline{s}_{n-2} $ and $ \underline{s}_{n} $ in every reduced
	$ \underline{S} $-expression for every element 
	$a\in  P_2(W) $. Hence by Lemma \ref{Bl}, $ s_{n}\not\in D_{R}(a)$.
\end{proof}
Define $ P_3(W):=\Gamma(W)\backslash P_2(W) $. Note that $P_1(W)$, $P_2(W)$, and $P_3(W)$ are pairwise disjoint sets whose union is $\Di(W)$.
\begin{Lemma}\label{Claim}
	$ P_3(W)=\{y\underline{s}_{n}\underline{s}_{n-1}:y\in \Gamma(W_{-2})\} $.
\end{Lemma}
 
 \begin{proof}
	For any $ a\in \{y\underline{s}_{n}\underline{s}_{n-1}:y\in \Gamma(W_{-2})\} $
we have that $ a\not\in  P_2(W) $ 
since $ s_{n-2}\not\in D_{R}(y)$.
By Lemma \ref{L111}, $ \underline{s}_{n-1} $ comes after $ \underline{s}_{n-2} $
and $ \underline{s}_{n} $ in every reduced $\underline{S}$-expression
for $a$, hence $\underline{s}_{n}\leq a$ and $s_{n}\notin D_{R}(a)$.
Thus $a\in \Gamma(W)\backslash P_2(W)=P_3(W)$.

Conversely, assume that $a\in P_3(W)$. Since $\underline{s}_{n}\le a$ and $s_n\not \in D_R(a)$, 
$\underline{s}_{n-1}$ comes after $\underline{s}_n$ in every reduced $\underline{S}$-expression for $a$. 
Hence, no reduced $\underline{S}$-expression for $a$ can begin with $\underline{s}_n\underline{s}_{n-1}$, 
which means that $\underline{s}_{n-1}$ also comes after $\underline{s}_{n-2}$ in every reduced 
$\underline{S}$-expression for $a$. Therefore, $a = y\underline{s}_n\underline{s}_{n-1}$ for some 
$y\in \Delta_{\mathrm{inv}}(W_{-2})$. Moreover, $y\in \Gamma(W_{-2})$ because $a\not \in P_2(W)$.
\end{proof}

\begin{proof}[Proof of Theorem \ref{th:main}] 
Define matchings $ M_{1}, M_{2} $ and $ M_{3} $ on
  $ P_1(W), P_2(W) $ and $ P_3(W) $ respectively as:
  \begin{enumerate}
  	\item $ M_{1}(a):=a\underline{s}_{n} $ for
  	 $ a\in P_1(W) $;
  	\item $ M_2(a):=M^{-3}(x)\underline{s}_{n-2}\underline{s}_{n}
  	       \underline{s}_{n-1} $ for
  	       $ a=x\underline{s}_{n-2}\underline{s}_{n}
  	       \underline{s}_{n-1}\in P_2(W) $;
  	\item $ M_{3}(a):=M^{-2}(y)\underline{s}_{n}\underline{s}_{n-1} $
  	      for 
  	      $ a=y\underline{s}_{n}\underline{s}_{n-1}\in P_3(W) $.
  \end{enumerate}
Note that $M^{-2}(y)\in \Gamma(W_{-2})$ since $M^{-2}$ 
preserves $\Gamma(W_{-2})$. This ensures that $ M_{3} $ 
is well defined.
 
 Define also a matching

  \[ M(a):=\begin{cases}
  M_{1}(a) \: \text{if}\: a\in P_{1}(W),\\
  M_{2}(a) \: \text{if}\: a\in P_{2}(W),\\
  M_{3}(a) \: \text{if}\: a\in P_{3}(W)
  \end{cases} \] on  $P(\Di(W))$.
  We shall show that
  $ M $ is a $\Gamma$-matching on $P(\Di(W))$. 
  \\
  \\
   \textbf{A. $ M $ is an acyclic matching on $P(\Di(W))$}:

    We will first show that the matchings $  M_{1}$, $M_{2} $ and
    $ M_{3} $ are acyclic on $ P_{1}(W)$, $P_{2}(W) $ 
    and $ P_{3}(W) $
    respectively and then use
  Lemma \ref{L1} to show that $ M $ is an acyclic matching
  on $P(\Di(W))$.

   The matching $ M_{1} $ is acyclic. To see it:
By Lemma \ref{LL1}, if $M_{1}$ has a cycle on $P_{1}(W)$, then
$a_{0}\lessdot M_{1}(a_{0})=a_{0}\underline{s}_{n}
\gtrdot a_{1} $ and $ a_{1} \lessdot M_{1}(a_{1})=a_{1} \underline{s}_{n}$
 for some $ a_{0}\neq a_{1} $. However since 
$ \underline{s}_{n} \nleq a_{0}$ and 
$\underline{s}_{n}\nleq a_{1}$
 we have that $ a_{0}=a_{1} $, so there is no cycle. 
 Hence $M_{1}$ is acyclic.

Define a map
$g:P(\Di(W_{-3}))\rightarrow   P_{2}(W)$ by
$g(x)=x\underline{s}_{n-2}\underline{s}_{n}\underline{s}_{n-1}$. 
The subword property implies that $g$ is a poset isomorphism.  
Since $M^{-3}$  is acyclic, $M_{2}$ is an acyclic matching.

Similarly, $\Gamma(W_{-2})$  and
 $P_{3}(W)$ are isomorphic.
Since $M^{-2}$ is acyclic and $M^{-2}$ preserves $\Gamma(W_{-2})$, 
$M_{3}$  is an acyclic matching.

Our goal is to show that $ M $ is acyclic using Lemma \ref{L1}.
Let $ Q $ be the (total) order whose  Hasse diagram is as below:
	  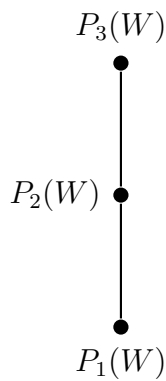
\begin{figure}[ht]
	  	\centering
	  	\begin{tikzpicture}[
	  	thick,
	  	acteur/.style={
	  		circle,
	  		fill=black,
	  		thick,
	  		inner sep=0.5pt,
	  		minimum size=0.2cm
	  	}
	  	]
	  	\node (a1) at (-1,0.75) [acteur,label=below:$P_{1}(W)$]{};
	  	\node (a2) at (-1,2.5) [acteur,label=left: $P_{2}(W) $]{};
	  	\node (a3) at (-1,4.25) [acteur,label=:$P_{3}(W)$]{};
	  	\draw (a1) -- (a2);
	  	\draw (a2) -- (a3);
	  	\end{tikzpicture}
	  	\caption{Hasse diagram of $ Q $.} \label{fi:bg(Q)}
	  \end{figure}

   We proceed to verify the hypotheses of the lemma.
	  \begin{enumerate}
	  	\item As we have already seen, each cell belongs to only one $ P_{i}(W) $ for $ i=1,2,3 $.
	  	\item We show that $ P_{1} (W)$ and $ P_{1}(W)\cup P_{2}(W) $
	  	are order ideals. To observe that $ P_{1} (W)$ is an
	  	order ideal we let $ a\in P_{1} (W)$ and $ t\leq a $. If $ \underline{s}_{n}\nleq a $,
	  	then $ \underline{s}_{n} \nleq t $ and hence $ t\in P_{1} (W)$. If $ s_{n} \in D_{R}(a)$,
	  	then either $ \underline{s}_{n}\nleq t $ or $ s_{n} \in D_{R}(t)$.
	  	So $ t\in P_{1} (W)$. Thus $ P_{1}(W) $ is an order ideal.
	  	
	  	Also $ P_{1}(W)\cup P_{2} (W)$ is an order ideal. 
	  	If $ a\in P_{2} (W)$ and $ t< a $, then we have
	  	$ t<a=x\underline{s}_{n-2}\underline{s}_{n}\underline{s}_{n-1} $.
	  	If $\underline{s}_{n-2}\underline{s}_{n}\underline{s}_{n-1}\leq  t $ then
	  	$ t\in P_{2} (W)$. If $ \underline{s}_{n-2}\underline{s}_{n}\underline{s}_{n-1}
	  	\nleq t $  then 
	  	 $  \underline{s}_{i} \nleq t $ for at least one of
	  	$ i=n-2,n-1,n $. If $\underline{s}_{n-2}\nleq t$ then $t=x\underline{s}_{n}
        \underline{s}_{n-1}=x\underline{s}_{n-1}
        \underline{s}_{n}\in P_{1}(W)$. Similarly if $\underline{s}_{n}$ or
        $\underline{s}_{n-1}\nleq t$, then $t=x\underline{s}_{n-2}
        \underline{s}_{n-1}\in P_{1}(W)$ or $t=x\underline{s}_{n-2}\underline{s}_{n}\in P_{1}(W)$.
	    Hence $ P_{1}(W)\cup P_{2} (W)$ is an order ideal.
	  \end{enumerate}
   Now since $ P_{1}(W), P_{1}(W)\cup P_{2}(W) $ and $ P_{1}(W)\cup P_{2}(W)\cup P_{3} (W)$
are order ideals, and $ M_{1}, M_{2} $ and $ M_{3} $ are acyclic matchings on
$ P_{1}(W), P_{2} (W)$, and $ P_{3} (W)$ respectively, then by Lemma \ref{L1}, $M$ is
an acyclic matching on $ P(\Di(W)) $.
\\
\textbf{B. All critical cells are of top dimension $ n-1 $}:
Let $ a $ be a critical cell of $ M $. Since $ M_{1} $ 
is complete on $ P_{1}(W) $, then either
$ a = x\underline{s}_{n-2}\underline{s}_{n}\underline{s}_{n-1} $ with $ x $ a critical cell of $ M^{-3} $
or $ a =y\underline{s}_{n}\underline{s}_{n-1} $ with $ y $ a critical cell of $ M^{-2} $. 
By the assumptions on 
$ M^{-2} $ and $M^{-3}$,  $\text{dim}(a)=n-1$ in both cases. 
\\
\textbf{C. $ M $ preserves $ \Gamma(W) $}:
	That the matching $ M $ preserves $ \Gamma(W) $ follows by construction since
	$ \Gamma(W) = P_{2}(W) \cup P_{3}(W)$.\\
	\textbf{D. All critical cells belong to $ \Gamma(W) $}: Since
	 $ M $ is complete on $ P_{1} (W)$,
	 there are no critical cells in $ P_{1} (W)$. So 
	 all critical cells of $M$ are in
	  $P(\Di(W))\backslash P_{1} (W)=\Gamma(W)$.
	
This completes the proof that $M$
is a $\Gamma$-matching on $P(\Di(W))$. 
Finally, the assertion about the number of critical cells is 
immediate from part B above.
\end{proof}	

\section{Applications of Theorem \ref{th:mai2}} \label{sec:mainn}

In this section we shall calculate the homotopy types of some explicit boolean involution complexes, 
including all those of finite Coxeter groups. 
Barring some small examples that are easily computed by hand, 
we achieve this by identifying $\Gamma$-matchings of some suitably chosen ordered Coxeter systems, 
and then invoking Theorem \ref{th:mai2}. 

\begin{Theorem}\label{Theorem}\label{th:maii1}
If $W$ is irreducible and finite, the homotopy type of $\Di(W)$ is as 
in the following table  

\begin{center}
\begin{table}[H]
\begin{tabular}{|c|c|} 
 \hline
\emph{Coxeter group $W$} & \emph{Homotopy type  of $\Di(W)$}\\ 
\hline
$A_{n}$ & $ \beta_I(A_{n})\cdot S^{n-1} $  \\ 
 $B_{n}$ &$\beta_I(B_{n})\cdot S^{n-1}$  \\ 
 $D_{n}$ &$\beta_I(D_{n})\cdot S^{n-1}$  \\
 $E_{6}$ & \emph{a point}  \\ 
 $E_{7}$ &$S^{6}$  \\ 
 $E_{8}$ &$S^{7}$  \\ 
 $F_{4}$ &\emph{a point}\\
 $H_{3}$ &$S^{2}$ \\ 
 $H_{4}$ &$S^{3}$ \\
 $I_{2}(m)$ &$S^{1}$ \\
 \hline
\end{tabular}
\vspace{1 cm}
\end{table}
\end{center}
\setlength{\tabcolsep}{-2pt}
where the Betti numbers in the classical types are determined by:
	\begin{itemize}
  	\item $\beta_I(A_{n})=\beta_I(A_{n-2})+\beta_I(A_{n-3})$ with
  	$\beta_I(A_{1})=0, \beta_I(A_{2})=0$ and $ \beta_I(A_{3})=1$,
  	\item $\beta_I(B_{n})=\beta_I(B_{n-2})+\beta_I(B_{n-3})$
   with 
  	$\beta_I(B_{2})=1, \beta_I(B_{3})=1$ and $ \beta_I(B_{4})=1$,
  	\item $\beta_I(D_{n})=\beta_I(D_{n-2})+\beta_I(D_{n-3})$  with 
  	$\beta_I(D_{4})=1, \beta_I(D_{5})=1 $
  	 and $\beta_I(D_{6})=2$.
  	\end{itemize}	
\end{Theorem}
\begin{Remark}
Recall that the {\em Padovan sequence} $P_{0}, P_{1}, \ldots$ is determined by $P_{0}=1, P_{1}=0, P_{2}=0$ 
and $P_{n}=P_{n-2}+P_{n-3}$ for $n\geq 3$. Hence, $\beta_I(A_{n})=P_{n}$ for $n\geq 1$, 
$\beta_I(B_{n})=P_{n+3}$ for $n\geq 2$, and $\beta_I(D_{n})=P_{n+2}$ for $n\geq 2$. For more information about 
the Padovan sequence see \cite{oeisA000931} and the references cited there.
\end{Remark}
 \begin{figure}[ht]
		\includegraphics[width=\linewidth]{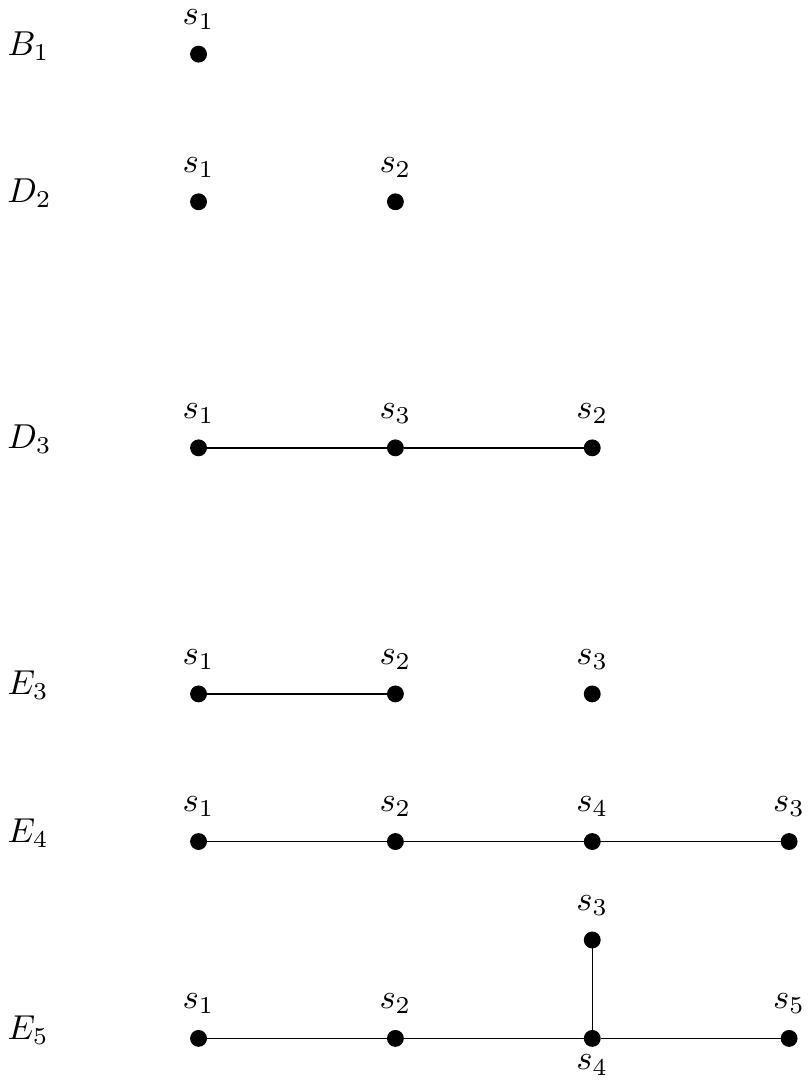}
		\vspace{-10cm}
		\caption{Ordered Coxeter systems that appear in the proof of Theorem \ref{th:maii1}}
		\label{fig:DEn}
	\end{figure}	
\setlength\belowcaptionskip{-3ex}

\begin{proof}[Proof of Theorem \ref{th:maii1}]
We provide $ \Gamma $-matchings of $ P(\Di(W))$. In types $A$, $B$, $D$, and $E$,
 this is achieved by using Theorem \ref{th:mai2}. Somewhat nonstandard type terminology 
 will be employed in order to indicate specific ordered Coxeter systems. 
 Their Coxeter graphs are collected in Figure \ref{fig:DEn}.

First, let us make a general observation. In order to verify that the boolean involution 
ideal of a certain ordered Coxeter system $(W,S)$ has a $\Gamma$-matching, it is enough 
to verify that the subposet induced by $\Gamma(W)$ admits an acyclic matching with all 
critical cells of top dimension. Namely, just as in the proof of Theorem \ref{th:main}, 
the subposet $P_1(W)$ is an order ideal with a complete acyclic matching, and so the 
union of these two matchings is a $\Gamma$-matching. 
This is illustrated for $W$ of type $F_{4}$ in Figure \ref{fig:F}. 
There, the entire acyclic matching on $P(\Di(W))$ is depicted. 
For the conclusion of the theorem, it is however enough to inspect 
$\Gamma(W)$, which consists of the two yellow elements.
\begin{figure}[ht]
		\includegraphics[width=\linewidth]{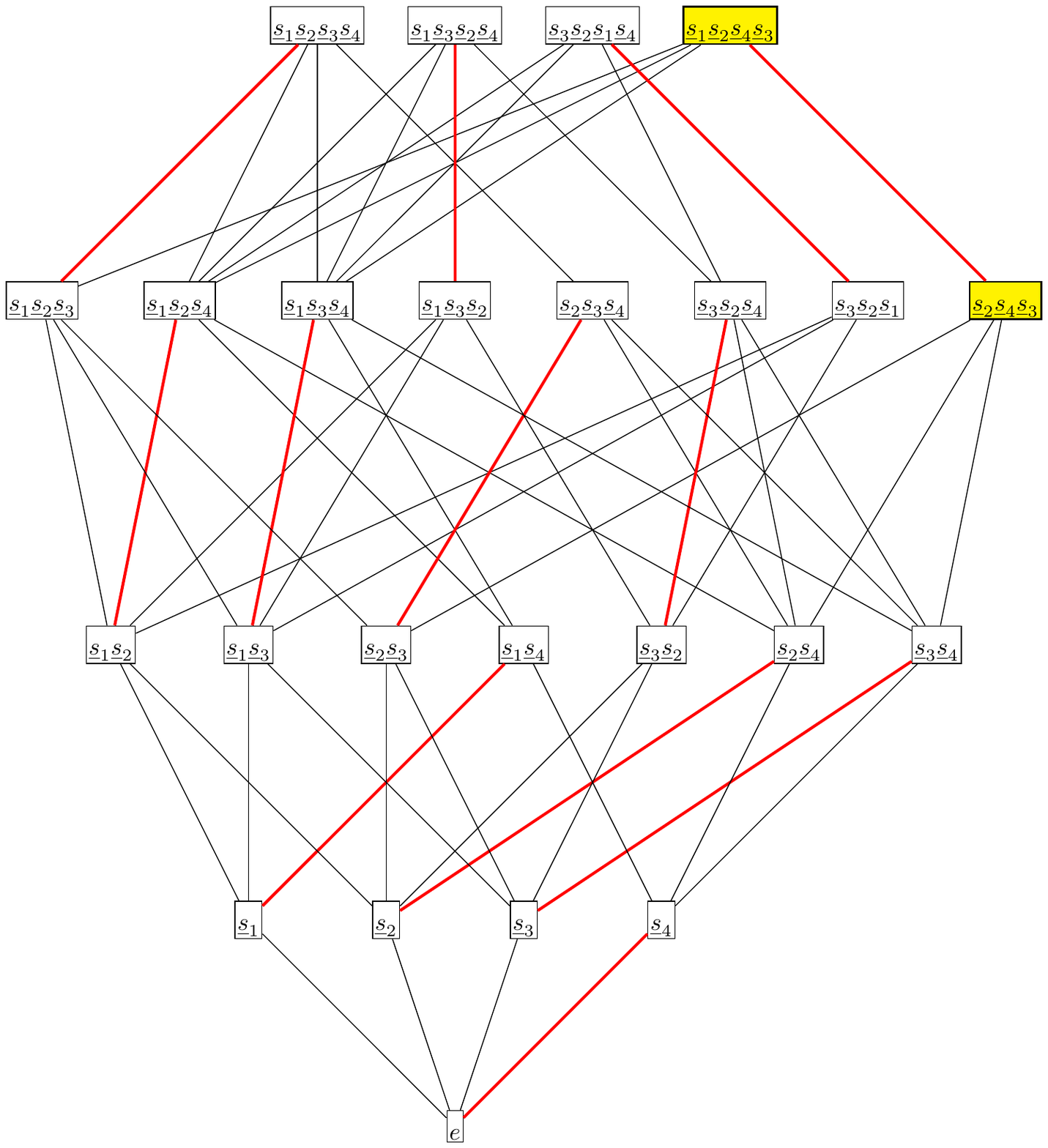}
		\vspace{-5cm}
		\caption{A $\Gamma$-matching on $P(\Di(F_{4}))$.}
		\label{fig:F}
	\end{figure}	
\setlength\belowcaptionskip{-3ex}
\newpage
\smallskip

\noindent {\bf Type A.} Consider $(W,S)$ of type $A_n$, $n \ge 4$, as an ordered Coxeter 
system with the natural order on the generators as indicated in Figure \ref{figs:d}. 
Then, $(W,S)$ is a path extension of the path ended system $A_4$. In order to obtain 
the desired result from Theorem \ref{th:mai2}, it therefore suffices to verify 
the hypotheses, namely that the boolean involution ideals of $A_{3}$, 
$A_{2}$,
 and $A_{1}$ have $\Gamma$-matchings. Concluding the proof in type $A$, 
 one readily verifies that 
 $\Gamma(A_{3}) = \{\underline{s}_1\underline{s}_3\underline{s}_2\}$ 
 and $\Gamma(A_{2}) = \Gamma(A_{1}) = \emptyset$.

\smallskip

\noindent {\bf Type B.} We proceed similarly to type $A$. This time, we begin with the 
ordered type $B$ system in Figure \ref{figs:d} which, for $n\ge 4$, is a path extension 
of $B_4$. Here, $\Gamma(B_{3}) = \{\underline{s}_1\underline{s}_3\underline{s}_2\}$, 
$\Gamma(B_{2}) = \{\underline{s}_2\underline{s}_1\}$, and $\Gamma(B_{1}) = \emptyset$.

\smallskip

\noindent {\bf Type D.} Again, we argue in a similar way. Here, we begin
 with the ordered system of type $D_{n}$ as in Figure \ref{figs:d}, where 
 $D_{n}$ is a path extension of $D_{5}$ for all $ n\geq 5$.
 The hypotheses are verified by observing that
 $\Gamma(D_{2}) = \emptyset$,
 $\Gamma(D_{3}) = \{\underline{s}_1\underline{s}_3\underline{s}_2\}$, 
 and $\Gamma(D_{4}) =\lbrace \underline{s}_1\underline{s}_4\underline{s}_3,$
 $\underline{s}_2\underline{s}_4\underline{s}_3,$
 $\underline{s}_1\underline{s}_4\underline{s}_3\underline{s}_2, $
 $\underline{s}_1\underline{s}_2\underline{s}_4\underline{s}_3,  $
 $\underline{s}_2\underline{s}_4\underline{s}_3\underline{s}_1\rbrace$. 
 Here, an acyclic matching on $\Gamma(D_{4})$ is depicted in Figure \ref{fig:D4}
 where the only critical cell is 
 $\underline{s}_2\underline{s}_4\underline{s}_3\underline{s}_1$.
 \begin{figure}[ht]
		\includegraphics[width=\linewidth]{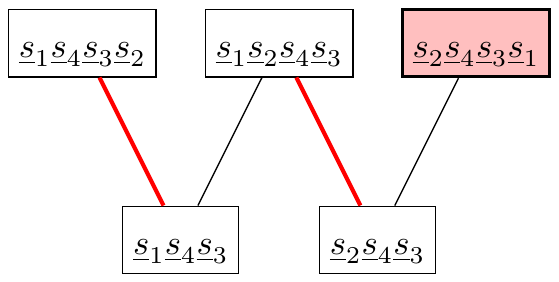}
		\vspace{-15cm}
		\caption{An acyclic matching on $\Gamma(D_{4})$.}
		\label{fig:D4}
	\end{figure}	
\setlength\belowcaptionskip{-3ex}
\smallskip

\noindent {\bf Type E.} 
We proceed in a similar way as in types $A$, $B$, $D$ above. 
We start with the ordered system of 
type $E$ as in Figure \ref{figs:d} where $E_{6}$,  $E_{7}$, and $E_{8}$ are 
path extensions of $E_{6}$.  
The hypotheses are readily checked by 
observing that $\Gamma(E_{3}) = \emptyset$,
 $\Gamma(E_{4}) = \{\underline{s}_2\underline{s}_4\underline{s}_3$, 
   $\underline{s}_1\underline{s}_4\underline{s}_2$,
   $\underline{s}_1\underline{s}_2\underline{s}_4\underline{s}_3$,
 $\underline{s}_1\underline{s}_4\underline{s}_2\underline{s}_3\}$ see Figure \ref{fig:E4}, 
 and $\Gamma(E_{5})$ is as depicted in Figure \ref{fig:E5} where an acyclic 
 matching on $\Gamma(E_{5})$ with one critical 
 cell,  
$\underline{s}_3\underline{s}_5\underline{s}_4\underline{s}_2\underline{s}_1$, is indicated. 
 \begin{figure}[ht]
 \centering
		\includegraphics[width=\linewidth]{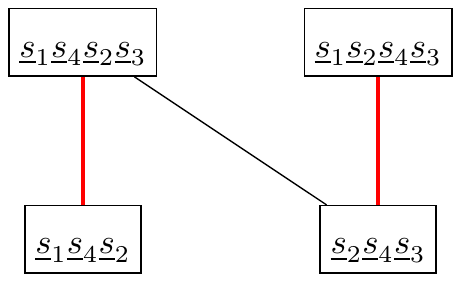}
		\vspace{-14cm}
		\caption{An acyclic matching on $\Gamma(E_{4})$.}
		\label{fig:E4}
	\end{figure}	
\setlength\belowcaptionskip{-3ex}
\newpage

\smallskip

 \begin{figure}[ht]
 \centering
		\includegraphics[width=\linewidth]{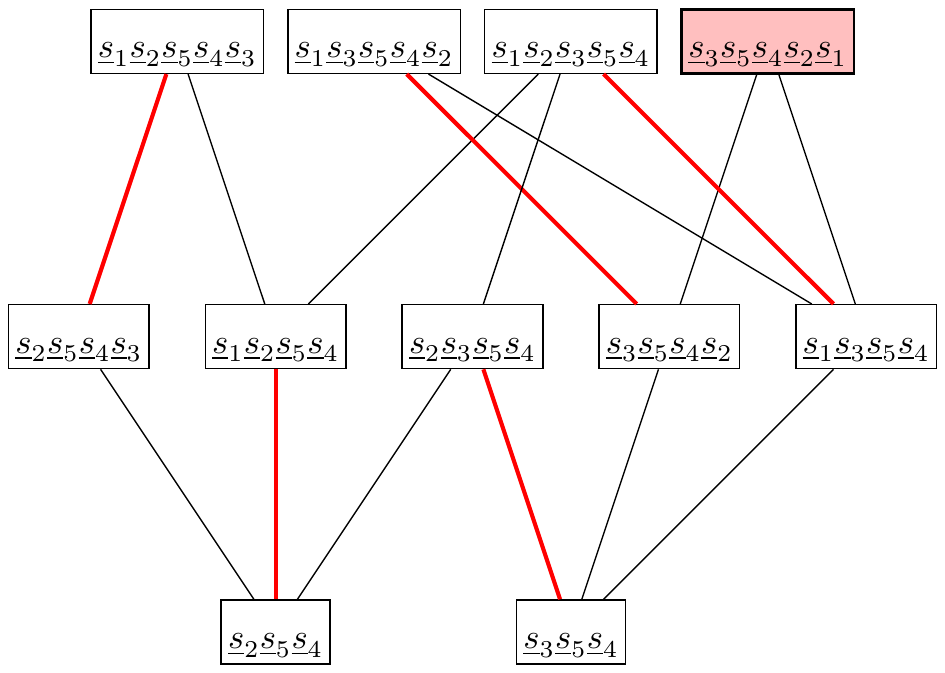}
		\vspace{-12cm}
		\caption{An acyclic matching on $\Gamma(E_{5})$.}
		\label{fig:E5}
	\end{figure}	
\setlength\belowcaptionskip{-3ex}
\newpage

\smallskip

\noindent {\bf Type H.}
The results in type $H$ are the same as in type $B$ 
by Lemma \ref{le:r}. 
\end{proof}
\newpage
Some concluding remarks are in order:

First, the previous theorem concerns irreducible groups, but reducible groups are easily covered. Namely, taking products of Coxeter systems is readily seen to imply taking the join of their corresponding boolean complexes of involutions.

Second, there is of course nothing that prevents extending the approach of the previous proof to path extensions which are not finite. For example,
$\tilde{F}_{4}$ is a path extension of $F_{4}$ and $\Di(\tilde{F}_{4})$ 
is homotopy equivalent to $S^{4}$, 
$\tilde{E}_{8}$ is a path extension of $E_{8}$ and $\Di(\tilde{E}_{8})$ 
is homotopy equivalent to $S^{8}$, etc.

Third, just as we did for type $H$ in the proof above, we may apply Lemma \ref{le:r} to obtain results for more Coxeter systems without additional effort.

Let us end with the natural question: Does the above result extend to all Coxeter systems? That is, is $\Di(W)$ homotopy equivalent to a wedge of spheres of dimension $|S|-1$ for every system $(W,S)$?

\bibliographystyle{amsplain}
\bibliography{boolean}
\end{document}